\documentclass[11pt]{article}

\usepackage[a4paper,margin=1in]{geometry}
\usepackage[T1]{fontenc}
\usepackage[utf8]{inputenc}
\usepackage{lmodern}
\usepackage{microtype}
\usepackage{amsmath,amssymb,amsthm,mathtools}
\usepackage{booktabs,longtable,array}
\usepackage{enumitem}
\usepackage[hidelinks]{hyperref}

\allowdisplaybreaks

\newtheorem{theorem}{Theorem}[section]
\newtheorem{proposition}[theorem]{Proposition}
\newtheorem{corollary}[theorem]{Corollary}
\newtheorem{lemma}[theorem]{Lemma}
\theoremstyle{definition}
\newtheorem{definition}[theorem]{Definition}
\newtheorem{problem}[theorem]{Problem}
\theoremstyle{remark}
\newtheorem{remark}[theorem]{Remark}

\newcommand{\Supp}{\operatorname{Supp}}
\newcommand{\prof}{\operatorname{prof}}
\newcommand{\mult}{\operatorname{mult}}
\newcommand{\gapv}{\operatorname{gap}}
\newcommand{\degG}{\operatorname{deg}}

\title{Support and Support Jumps in the Partition Graph}
\author{Fedor B. Lyudogovskiy}
\date{}

\begin{document}
\maketitle

\begin{abstract}
Let $G_n$ be the partition graph whose vertices are the partitions of $n$, with adjacency given by elementary transfers of one cell between parts, followed by reordering. In this paper we isolate the support of a partition, that is, the set of distinct part sizes, and treat it as a global invariant of $G_n$.

We introduce support size, ordered support profile, multiplicity pattern, and gap pattern, and show that support size $r$ occurs in $G_n$ if and only if $T_r=r(r+1)/2\le n$. Equivalently, the maximal support size realized in $G_n$ is $\rho(n)=\max\{r:T_r\le n\}$. We then determine exactly how the support changes along an edge of $G_n$: the support jump is always one of
\[
-2,-1,0,1,2,
\]
and an explicit birth--death formula is obtained in terms of the source and target part sizes.

Using the degree formula from the local morphology line, we show that
\[
\degG(\lambda)\ge \sigma(\lambda)(\sigma(\lambda)-1)
\]
for every partition $\lambda$, with equality exactly for staircase partitions. Thus support already captures a universal quadratic baseline for degree. We also prove that support size is invariant under conjugation, identify the support-$1$ stratum with the rectangular partitions, and show that the coarse support-level graph always contains the chain
\[
1-2-\cdots-\rho(n).
\]

Finally, we present an initial computational atlas for small $n$, including support-stratum counts, support-jump counts, and first connectivity data for the induced fixed-support subgraphs. The paper is intended as a bridge between the local block-morphology theory and a later study of jump-type and gradient-type invariants on the partition graph.
\end{abstract}

\noindent\textbf{Keywords.} integer partitions; partition graph; support of a partition; support jumps; degree bounds; conjugation symmetry; support corridors.


\section{Introduction}

In the partition graph $G_n$, the vertices are the partitions of $n$, and two vertices are adjacent when one partition is obtained from the other by an elementary transfer of one cell between parts, followed by reordering. Transfer graphs on partitions, or closely related models, have appeared previously in the study of Gray codes and minimal-change generation for partitions; see, for example, Savage~\cite{Savage1989}, Rasmussen--Savage--West~\cite{RasmussenSavageWest1995}, and the survey of M"utze~\cite{Mutze2023}. A different partition-graph model based on binary words was studied by Bal~\cite{Bal2022}. For the topological side of the same transfer graph considered here, see the recent clique-complex paper~\cite{LyudogovskiyHomotopy2026}. This graph has turned out to admit several mutually interacting geometric descriptions. In earlier parts of the project, local invariants such as degree, local clique number, and local simplex dimension were expressed through local transfer data, while global papers developed the language of the framework, the self-conjugate axis, the central region, simplicial shells, directional anisotropy, and morphogenesis across $n$.

The present paper isolates another natural invariant of a partition vertex, namely its support: the set of distinct part sizes occurring in the partition. At a formal level, the support is already implicit in the block decomposition
\[
\lambda=s_1^{m_1}s_2^{m_2}\cdots s_r^{m_r},
\qquad
s_1>\cdots>s_r\ge 1,
\]
where the $s_i$ are the distinct part sizes and the $m_i$ record their multiplicities. However, in the earlier local theory this information appears only as one component inside a richer transfer-type package. The aim of the present paper is different: we treat support as a primary global invariant in its own right and begin a systematic study of its behaviour across the whole graph $G_n$.

There are several reasons why support deserves such a separate treatment. First, the support size
\[
\sigma(\lambda):=|\Supp(\lambda)|
\]
measures a basic form of combinatorial diversity of the partition. Second, support behaves nontrivially under elementary transfers, and therefore gives rise not only to a vertex statistic but also to an edge statistic. Third, support interacts naturally with several geometric structures already present in the project: staircase and near-staircase behaviour, rectangular families, conjugation symmetry, and the distinction between locally simple and locally complex regions of the graph. Finally, support is a natural precursor to a later theory of jump-type and gradient-type invariants: before studying more elaborate transfer gradients, one should first understand how one of the most basic block parameters changes along edges.

Our first structural result concerns the possible support sizes in $G_n$. If
\[
T_r=\frac{r(r+1)}2
\]
denotes the $r$-th triangular number, then support size $r$ occurs in $G_n$ if and only if $T_r\le n$. Equivalently, the maximal support size realized among partitions of $n$ is
\[
\rho(n):=\max\{r:T_r\le n\}.
\]
Thus the support strata
\[
V_{n,r}:=\{\lambda\vdash n:\sigma(\lambda)=r\}
\]
exist exactly for $1\le r\le \rho(n)$.

The second main result is an exact edgewise support-change theorem. If an edge of $G_n$ is produced by moving one cell from a part of size $x$ to a part of size $y$, then the support can change only through deaths at the sizes $x$ and $y$ and births at the sizes $x-1$ and $y+1$, with one collision correction when the two candidate births coincide. As a consequence, the support jump along an edge is always one of
\[
-2,-1,0,1,2,
\]
and all five values do occur.

A third structural point is that support already controls degree at the first approximation level. Using the degree formula established in the local morphology line, we show that
\[
\degG(\lambda)\ge \sigma(\lambda)(\sigma(\lambda)-1)
\]
for every partition $\lambda$, with equality exactly for staircase partitions. Thus support alone provides a universal quadratic floor for degree, while multiplicity and enlarged support gaps contribute only nonnegative correction terms.

The support is also well behaved under conjugation. If $\lambda'$ denotes the conjugate partition of $\lambda$, then $\sigma(\lambda')=\sigma(\lambda)$. Hence every support stratum $V_{n,r}$ is invariant under the conjugation symmetry of $G_n$. At the low-support end, the stratum $V_{n,1}$ is exactly the family of rectangular partitions $(a^b)$. At the high-support end, the stratum $V_{n,\rho(n)}$ always contains a canonical staircase or near-staircase vertex.

These static and edgewise results lead naturally to a first corridor language. On the one hand, for each $r$ one may consider the induced subgraph on the fixed-support stratum $V_{n,r}$, which records support-preserving motion. On the other hand, one may compress the graph to its support levels and study the coarse support-level graph on
\[
\{1,\dots,\rho(n)\}.
\]
We prove that this coarse level graph is always connected and contains the chain
\[
1-2-\cdots-\rho(n).
\]

Methodologically, the paper is intended as a bridge paper. It does not attempt to replace the finer local transfer theory, and it does not yet develop a full jump-gradient formalism. Instead, it occupies an intermediate position between the two. From the local side, it extracts support from block decomposition and shows that this coarse invariant already carries meaningful structural information. From the future jump side, it introduces the language of support jumps, support-preserving edges, support-expanding and support-contracting transitions, and support corridors.

The paper is organized as follows. In Section~2 we introduce the support language and determine the possible support sizes in $G_n$. Section~3 gives the exact support-change theorem for a single edge. Section~4 studies the relation between support and degree. Section~5 turns to global support distributions, conjugation symmetry, and the first exact descriptions of low- and high-support strata. Section~6 introduces support corridors and the coarse support-level graph. Section~7 gives an initial computational atlas. Section~8 concludes with open problems leading toward the jump/gradient line.

\section{Support language and support strata}

\subsection{Support, support size, and derived block data}

Let $\lambda\vdash n$. Write $\lambda$ in block form as
\[
\lambda=s_1^{m_1}s_2^{m_2}\cdots s_r^{m_r},
\qquad
s_1>s_2>\cdots>s_r\ge 1,
\qquad
m_i\ge 1.
\]

\begin{definition}
For a partition $\lambda\vdash n$, define:
\begin{enumerate}[label=\arabic*.]
\item the \emph{support}
\[
\Supp(\lambda):=\{s\ge 1:\text{ the part size }s\text{ occurs in }\lambda\};
\]
\item the \emph{support size}
\[
\sigma(\lambda):=|\Supp(\lambda)|;
\]
\item the \emph{ordered support profile}
\[
\prof(\lambda):=(s_1,\dots,s_r);
\]
\item the \emph{multiplicity pattern}
\[
\mult(\lambda):=(m_1,\dots,m_r);
\]
\item the \emph{gap pattern}
\[
\gapv(\lambda):=(g_1,\dots,g_{r-1}),
\qquad
g_i:=s_i-s_{i+1}.
\]
\end{enumerate}
\end{definition}

The support and support size are the primary invariants of the paper. The multiplicity and gap patterns are derived refinement data attached to the ordered support profile. For standard background on integer partitions, see Andrews~\cite{Andrews1976} and Stanley~\cite{Stanley2011}.

\begin{definition}
For $n\ge 1$ and $r\ge 1$, define the \emph{support stratum}
\[
V_{n,r}:=\{\lambda\vdash n:\sigma(\lambda)=r\}.
\]
\end{definition}

\subsection{Extremal support}

\begin{lemma}
Let $\lambda\vdash n$, and suppose that $\sigma(\lambda)=r$. Then
\[
n\ge T_r:=\frac{r(r+1)}2.
\]
Moreover, equality holds if and only if
\[
\lambda=(r,r-1,\dots,2,1).
\]
\end{lemma}

\begin{proof}
A partition with $r$ distinct positive part sizes must have total size at least
\[
1+2+\cdots+r=T_r.
\]
Equality occurs only when the distinct part sizes are exactly $1,2,\dots,r$, each with multiplicity $1$.
\end{proof}

\begin{lemma}
If $n\ge T_r$, then
\[
\lambda^{\max}_{n,r}:=(r+(n-T_r),\,r-1,\,r-2,\,\dots,\,2,\,1)
\]
is a partition of $n$ with support size $r$.
\end{lemma}

\begin{proof}
Its parts are distinct and sum to $n$, so its support size is exactly $r$.
\end{proof}

\begin{theorem}
Let $n\ge 1$ and $r\ge 1$. Then $V_{n,r}\neq\varnothing$ if and only if
\[
T_r\le n.
\]
Consequently,
\[
\max_{\lambda\vdash n}\sigma(\lambda)
=
\rho(n):=\max\{r:T_r\le n\}.
\]
\end{theorem}

\begin{proof}
The necessity is the previous lemma on the minimum possible total size for $r$ distinct part sizes. The sufficiency is realized by $\lambda^{\max}_{n,r}$.
\end{proof}

\begin{corollary}
Support size $r$ occurs for the first time at $n=T_r$, and at that value the unique partition with support size $r$ is the staircase
\[
(r,r-1,\dots,2,1).
\]
\end{corollary}

\section{Support change along an edge}

We now turn from support as a static invariant to support as a transition quantity.

\subsection{Transfer language}

Let $\lambda\sim\mu$ be an edge of $G_n$. Then $\mu$ is obtained from $\lambda$ by moving one cell from a part of size $x$ to a part of size $y$, where $y=0$ is allowed and corresponds to creating a new row. Thus, at the multiset level,
\[
\mu
=
\bigl(\lambda\setminus\{x,y\}\bigr)\cup\{x-1,y+1\},
\]
with the obvious conventions when $y=0$ or $x=1$.

\begin{definition}
For an oriented edge $\lambda\to\mu$, define the \emph{support jump}
\[
\Delta_\sigma(\lambda,\mu):=\sigma(\mu)-\sigma(\lambda).
\]
\end{definition}

\subsection{Multiplicity update}

\begin{lemma}
Let $\mu$ be obtained from $\lambda$ by transferring one cell from a part of size $x$ to a part of size $y$, with $y=0$ allowed. Then for every $t\ge 1$,
\[
m_t(\mu)
=
m_t(\lambda)
-\mathbf 1_{t=x}
-\mathbf 1_{t=y}
+\mathbf 1_{t=x-1}
+\mathbf 1_{t=y+1},
\]
where the term $\mathbf 1_{t=y}$ is omitted when $y=0$, and the term $\mathbf 1_{t=x-1}$ is omitted when $x=1$.
\end{lemma}

\begin{proof}
Exactly one part of size $x$ is removed, exactly one part of size $y$ is removed when $y>0$, exactly one part of size $x-1$ is created when $x>1$, and exactly one part of size $y+1$ is created. No other multiplicities change.
\end{proof}

\begin{corollary}
Only the support values among
\[
x,\ y,\ x-1,\ y+1
\]
can be affected by a single elementary transfer.
\end{corollary}

A support value can die only at $x$ or $y$, and can be born only at $x-1$ or $y+1$. The only collision occurs when
\[
x-1=y+1,
\qquad\text{i.e.}\qquad
x=y+2.
\]

\begin{theorem}[Exact support change along an edge]
Let $\lambda\sim\mu$ be an edge of $G_n$, obtained by moving one cell from a part of size $x$ to a part of size $y$, with $y=0$ allowed, and let $S=\Supp(\lambda)$.

If $x\neq y$, then
\[
\Delta_\sigma(\lambda,\mu)
=
\mathbf 1_{x>1,\ x-1\notin S}
+
\mathbf 1_{y+1\notin S}
-
\mathbf 1_{x=y+2,\ x-1\notin S}
-
\mathbf 1_{m_x(\lambda)=1}
-
\mathbf 1_{y>0,\ m_y(\lambda)=1}.
\]

If $x=y$, then necessarily $m_x(\lambda)\ge 2$, and
\[
\Delta_\sigma(\lambda,\mu)
=
\mathbf 1_{x>1,\ x-1\notin S}
+
\mathbf 1_{x+1\notin S}
-
\mathbf 1_{m_x(\lambda)=2}.
\]
\end{theorem}

\begin{proof}
The support jump is the number of births minus the number of deaths. Death at $x$ occurs exactly when $m_x(\lambda)=1$; death at $y$ occurs exactly when $y>0$ and $m_y(\lambda)=1$. Birth at $x-1$ occurs exactly when $x>1$ and $x-1\notin S$; birth at $y+1$ occurs exactly when $y+1\notin S$. If $x\neq y$, the only overcounting is the collision case $x=y+2$. If $x=y$, the original size $x$ disappears exactly when both copies are used up, i.e. when $m_x(\lambda)=2$.
\end{proof}

\begin{corollary}
For every edge of $G_n$,
\[
\Delta_\sigma\in\{-2,-1,0,1,2\}.
\]
\end{corollary}

\begin{corollary}
Every value in $\{-2,-1,0,1,2\}$ occurs for suitable edges in suitable graphs $G_n$.
\end{corollary}

\begin{proof}
Examples are
\[
(4,4,1,1)\to(4,3,2,1)\quad(+2),
\]
\[
(5,5)\to(6,4)\quad(+1),
\]
\[
(4,1)\to(3,2)\quad(0),
\]
\[
(4,2)\to(3,3)\quad(-1),
\]
\[
(3,2,1)\to(3,3)\quad(-2).
\]
\end{proof}

\section{Support versus degree}

We now make precise one of the main bridge phenomena behind the paper: support does not determine the degree, but it provides a universal quadratic baseline for it.

\subsection{Augmented gap data and the imported degree formula}

Let
\[
\lambda=s_1^{m_1}s_2^{m_2}\cdots s_r^{m_r},
\qquad r=\sigma(\lambda).
\]
For degree-theoretic purposes we use the \emph{augmented gap sequence}
\[
\widehat g_i:=s_i-s_{i+1}\qquad (1\le i\le r),
\]
where by convention $s_{r+1}:=0$.

\begin{remark}
This augmented gap convention is finer than the internal gap pattern from Section~2, since it includes the terminal drop from the smallest supported part size down to $0$.
\end{remark}

\begin{proposition}[Imported local degree formula]
\[
\degG(\lambda)
=
r(r-1)
+
\sum_{i=1}^r \mathbf 1_{m_i>1}
+
\sum_{i=1}^r \mathbf 1_{\widehat g_i>1}.
\]
\end{proposition}

Thus
\[
\degG(\lambda)
=
\sigma(\lambda)(\sigma(\lambda)-1)
+
E_{\mathrm{mult}}(\lambda)
+
E_{\mathrm{gap}}(\lambda),
\]
where both correction terms are nonnegative.

\begin{corollary}
For every partition $\lambda\vdash n$,
\[
\degG(\lambda)\ge \sigma(\lambda)(\sigma(\lambda)-1).
\]
\end{corollary}

\begin{proposition}
For a partition $\lambda$, the following are equivalent:
\begin{enumerate}[label=\arabic*.]
\item
\[
\degG(\lambda)=\sigma(\lambda)(\sigma(\lambda)-1);
\]
\item
\[
m_i=1 \text{ for all }i,
\qquad
\widehat g_i=1 \text{ for all }i;
\]
\item $\lambda$ is a staircase partition:
\[
\lambda=(r,r-1,\dots,2,1),
\qquad r=\sigma(\lambda).
\]
\end{enumerate}
\end{proposition}

\begin{proof}
Equality holds exactly when both correction terms vanish. Then $s_r=1$, $s_{r-1}=2$, and so on, hence the support is $\{r,r-1,\dots,1\}$, and all multiplicities are $1$.
\end{proof}

\begin{corollary}
If $\sigma(\lambda)=r$ and $|\lambda|>T_r$, then
\[
\degG(\lambda)\ge r(r-1)+1.
\]
\end{corollary}

Formula above shows that support gives the \emph{quadratic skeleton}
\[
\sigma(\lambda)(\sigma(\lambda)-1),
\]
while repeated part sizes and widened support gaps contribute only binary correction terms. In this sense, support is already a meaningful coarse organizer of the degree landscape.

\section{Support distributions, conjugation, and basic localization}

\subsection{Support enumerators}

Let
\[
a_{n,r}:=|V_{n,r}|
\]
and define
\[
S_n(u):=\sum_{\lambda\vdash n}u^{\sigma(\lambda)}
=\sum_{r=1}^{\rho(n)}a_{n,r}u^r.
\]

\begin{proposition}
\[
\sum_{r=1}^{\rho(n)}a_{n,r}=p(n),
\]
where $p(n)$ is the partition number.
\end{proposition}

\subsection{Conjugation preserves support size}

\begin{proposition}
Let
\[
\lambda=s_1^{m_1}\cdots s_r^{m_r},
\qquad
M_j:=m_1+\cdots+m_j.
\]
Then
\[
\Supp(\lambda')=\{M_1,\dots,M_r\}.
\]
In particular,
\[
\sigma(\lambda')=\sigma(\lambda).
\]
\end{proposition}

\begin{proof}
If
\[
\lambda=s_1^{m_1}\cdots s_r^{m_r},
\]
then the distinct column heights in the Ferrers diagram are precisely the partial sums $M_1,\dots,M_r$, which are therefore the distinct part sizes of $\lambda'$.
\end{proof}

\begin{corollary}
Each support stratum $V_{n,r}$ is invariant under conjugation.
\end{corollary}

\subsection{The rectangle stratum}

\begin{proposition}
For a partition $\lambda\vdash n$, the following are equivalent:
\begin{enumerate}[label=\arabic*.]
\item $\sigma(\lambda)=1$;
\item $\lambda=(a^b)$ for some positive integers $a,b$ with $ab=n$.
\end{enumerate}
Consequently,
\[
a_{n,1}=d(n),
\]
the number of positive divisors of $n$.
\end{proposition}

\begin{corollary}
A support-$1$ partition is self-conjugate if and only if it is a square:
\[
\lambda=(a^a),
\qquad n=a^2.
\]
\end{corollary}

\subsection{Support-maximal witnesses}

\begin{proposition}
Let $r=\rho(n)$, and write
\[
n=T_r+t,\qquad t\ge 0.
\]
Then
\[
\lambda_n^{\max}:=(r+t,r-1,r-2,\dots,2,1)
\]
belongs to $V_{n,r}$.
\end{proposition}

Thus every $G_n$ contains a canonical support-maximal staircase or near-staircase vertex.

\section{Support corridors and jump structure}

We now pass from support distributions to the geometry of paths and transitions.

\subsection{Support-preserving subgraphs and the support-level graph}

For $1\le r\le \rho(n)$, let
\[
G_n^{[r]}
\]
be the induced subgraph of $G_n$ on $V_{n,r}$.

\begin{definition}
The \emph{support-level graph} $\mathcal L_n^\sigma$ is the graph with vertex set
\[
\{1,2,\dots,\rho(n)\},
\]
in which $r$ and $s$ are adjacent if there exists an edge $\lambda\mu\in E(G_n)$ such that
\[
\sigma(\lambda)=r,\qquad \sigma(\mu)=s.
\]
\end{definition}

\begin{proposition}
If $r$ and $s$ are adjacent in $\mathcal L_n^\sigma$, then
\[
|r-s|\le 2.
\]
\end{proposition}

\begin{proof}
This is the edgewise support-jump bound applied to an edge realizing that level transition.
\end{proof}

\subsection{The support-$1$ stratum}

The rectangle stratum is almost edgeless, but not quite.

\begin{proposition}
The induced subgraph $G_n^{[1]}$ consists of a single edge for $n=2$, namely
\[
(2)\sim(1,1),
\]
and is edgeless for all $n\ge 3$.
\end{proposition}

\begin{proof}
Every support-$1$ partition is a rectangle $(a^b)$. If $n=2$, the two rectangles are $(2)$ and $(1,1)$, and they are adjacent.

Assume now $n\ge 3$. Any nontrivial elementary transfer from a rectangle creates at least two distinct part sizes, so no edge joins two support-$1$ vertices.
\end{proof}

\subsection{Internal edges in higher strata}

\begin{proposition}
Let $r\ge 2$ and $n>T_r$. Then $G_n^{[r]}$ contains an edge.
\end{proposition}

\begin{proof}
Write $n=T_r+t$ with $t\ge 1$. Consider
\[
\lambda=(r+t,r-1,r-2,\dots,2,1)
\]
and
\[
\mu=(r+t-1,r-1,r-2,\dots,2,1,1).
\]
Both belong to $V_{n,r}$, and $\mu$ is obtained from $\lambda$ by moving one cell from the largest part to a new row.
\end{proof}

At $n=T_r$, this cannot happen because $V_{n,r}$ is then the singleton staircase partition.

\subsection{Consecutive support levels are always adjacent}

\begin{proposition}
Let $1\le r<\rho(n)$. Then there exists an edge of $G_n$ joining a vertex of $V_{n,r}$ to a vertex of $V_{n,r+1}$.
\end{proposition}

\begin{proof}
Write $n=T_{r+1}+t$ with $t\ge 0$. Consider
\[
\mu=(r+1+t,r,r-1,\dots,2,1)\in V_{n,r+1}
\]
and
\[
\lambda=(r+2+t,r,r-1,\dots,2)\in V_{n,r}.
\]
The partition $\lambda$ is obtained from $\mu$ by moving one cell from the part $1$ to the largest part, so $\lambda\sim\mu$.
\end{proof}

\begin{corollary}
The support-level graph $\mathcal L_n^\sigma$ contains the chain
\[
1-2-\cdots-\rho(n).
\]
In particular, $\mathcal L_n^\sigma$ is connected.
\end{corollary}

\subsection{Corridor language}

A path
\[
\lambda_0,\lambda_1,\dots,\lambda_m
\]
in $G_n$ is called:
\begin{enumerate}[label=\arabic*.]
\item a \emph{constant-support path} if all $\sigma(\lambda_i)$ are equal;
\item a \emph{support-monotone path} if the sequence $\sigma(\lambda_i)$ is weakly monotone;
\item a \emph{strict support-rise path} if $\sigma(\lambda_{i+1})>\sigma(\lambda_i)$ for all $i$.
\end{enumerate}

At the coarse level, the previous corollary provides a connected corridor theory on support levels. At the fine level, the connectivity of the induced graphs $G_n^{[r]}$ is subtler and belongs largely to the computational atlas.

\section{Computational atlas of support strata and support jumps}

In this section we record initial computational data for the support geometry of the partition graph. The purpose is twofold. First, the tables make the support stratification and the support-jump structure visible at small values of $n$. Second, they help separate theorem-level facts from empirical patterns that should not yet be promoted to formal statements.

All data in this section were obtained by direct enumeration of partitions of $n$, their neighbors in $G_n$, their support sizes, and the induced fixed-support subgraphs.

\subsection{Basic support-stratum counts}

Recall that
\[
a_{n,r}:=|V_{n,r}|.
\]
By Theorem~2.5, only the values $1\le r\le \rho(n)$ can occur.

The first support-stratum counts, for $1\le n\le 20$, are as follows.

\begin{center}
\small
\begin{longtable}{rccccc}
\toprule
$n$ & $a_{n,1}$ & $a_{n,2}$ & $a_{n,3}$ & $a_{n,4}$ & $a_{n,5}$\\
\midrule
\endfirsthead
\toprule
$n$ & $a_{n,1}$ & $a_{n,2}$ & $a_{n,3}$ & $a_{n,4}$ & $a_{n,5}$\\
\midrule
\endhead
1 & 1 & 0 & 0 & 0 & 0 \\
2 & 2 & 0 & 0 & 0 & 0 \\
3 & 2 & 1 & 0 & 0 & 0 \\
4 & 3 & 2 & 0 & 0 & 0 \\
5 & 2 & 5 & 0 & 0 & 0 \\
6 & 4 & 6 & 1 & 0 & 0 \\
7 & 2 & 11 & 2 & 0 & 0 \\
8 & 4 & 13 & 5 & 0 & 0 \\
9 & 3 & 17 & 10 & 0 & 0 \\
10 & 4 & 22 & 15 & 1 & 0 \\
11 & 2 & 27 & 25 & 2 & 0 \\
12 & 6 & 29 & 37 & 5 & 0 \\
13 & 2 & 37 & 52 & 10 & 0 \\
14 & 4 & 44 & 67 & 20 & 0 \\
15 & 4 & 44 & 97 & 30 & 1 \\
16 & 5 & 55 & 117 & 52 & 2 \\
17 & 2 & 59 & 154 & 77 & 5 \\
18 & 6 & 68 & 184 & 117 & 10 \\
19 & 2 & 71 & 235 & 162 & 20 \\
20 & 6 & 81 & 277 & 227 & 36 \\
\bottomrule
\end{longtable}
\end{center}

Several structural points are immediately visible.
\begin{enumerate}[label=\arabic*.]
\item The support range grows exactly as predicted by Theorem~2.5:
\[
\rho(n)=1,1,2,2,2,3,3,3,3,4,4,4,4,4,5,5,5,5,5,5
\]
for $1\le n\le 20$.
\item The lowest stratum is thin and highly arithmetic, in agreement with
\[
a_{n,1}=d(n).
\]
\item The support-maximal stratum begins as a singleton at its first occurrence $n=T_r$, and then expands rapidly.
\item In the range shown here, the bulk of the vertex set shifts rather quickly toward intermediate support values.
\end{enumerate}

\subsection{Support-jump counts on edges}

For undirected edges, define
\[
j_{n,\delta}
:=
\#\{\lambda\mu\in E(G_n): |\sigma(\lambda)-\sigma(\mu)|=\delta\},
\qquad
\delta\in\{0,1,2\}.
\]
By the edgewise jump bound, these are the only possible jump magnitudes.

The first jump counts are:

\begin{center}
\small
\begin{longtable}{rcccc}
\toprule
$n$ & $j_{n,0}$ & $j_{n,1}$ & $j_{n,2}$ & $|E(G_n)|$\\
\midrule
\endfirsthead
\toprule
$n$ & $j_{n,0}$ & $j_{n,1}$ & $j_{n,2}$ & $|E(G_n)|$\\
\midrule
\endhead
1 & 0 & 0 & 0 & 0 \\
2 & 1 & 0 & 0 & 1 \\
3 & 0 & 2 & 0 & 2 \\
4 & 1 & 4 & 0 & 5 \\
5 & 7 & 2 & 0 & 9 \\
6 & 7 & 8 & 2 & 17 \\
7 & 14 & 14 & 0 & 28 \\
8 & 19 & 26 & 2 & 47 \\
9 & 37 & 34 & 2 & 73 \\
10 & 54 & 52 & 8 & 114 \\
11 & 84 & 82 & 4 & 170 \\
12 & 119 & 118 & 16 & 253 \\
13 & 179 & 174 & 12 & 365 \\
14 & 245 & 252 & 28 & 525 \\
15 & 370 & 336 & 32 & 738 \\
16 & 491 & 486 & 56 & 1033 \\
17 & 698 & 666 & 58 & 1422 \\
18 & 940 & 900 & 108 & 1948 \\
19 & 1292 & 1226 & 116 & 2634 \\
20 & 1709 & 1650 & 186 & 3545 \\
\bottomrule
\end{longtable}
\end{center}

These data already show a fairly rich edge ecology.

\begin{enumerate}[label=\arabic*.]
\item Jump size $2$ first appears at $n=6$.
\item Support-preserving edges and support-changing edges are of comparable size from very early on.
\item In this range, jump size $1$ is the dominant nonzero jump, while jump size $2$ remains rarer but persistent.
\end{enumerate}

\subsection{The coarse support-level edge structure}

It is useful to refine the jump counts by recording how many edges run between specific support levels. Let
\[
e_{n,r,s}
:=
\#\{\lambda\mu\in E(G_n):\sigma(\lambda)=r,\ \sigma(\mu)=s\},
\qquad r\le s.
\]
Thus $e_{n,r,r}$ is the number of support-preserving edges inside $V_{n,r}$, and $e_{n,r,s}$ with $r<s$ measures the coupling between different support levels.

As a representative example, for $n=20$ the level-edge matrix is:

\[
\begin{array}{c|ccccc}
 & 1 & 2 & 3 & 4 & 5\\
\hline
1 & 0 & 4 & 6 & 0 & 0\\
2 & 4 & 41 & 214 & 76 & 0\\
3 & 6 & 214 & 606 & 980 & 104\\
4 & 0 & 76 & 980 & 942 & 452\\
5 & 0 & 0 & 104 & 452 & 120
\end{array}
\]

Several features are worth noting.
\begin{enumerate}[label=\arabic*.]
\item The matrix is concentrated near the diagonal, as expected from the bound
\[
|\Delta_\sigma|\le 2.
\]
\item The strongest inter-level interaction at $n=20$ occurs between levels $3$ and $4$, not between the lowest levels.
\item The extreme low-support level $1$ is only weakly connected to the rest of the graph, while the intermediate and upper-intermediate levels form a much denser transition zone.
\end{enumerate}

For later reference, the first occurrences of skip-two level couplings are
\[
(1,3)\text{ at }n=6,\qquad
(2,4)\text{ at }n=10,\qquad
(3,5)\text{ at }n=15,\qquad
(4,6)\text{ at }n=21.
\]
At present we record these only as computational facts.

\subsection{Fixed-support corridor data}

For each $r$, let $G_n^{[r]}$ denote the induced subgraph on $V_{n,r}$. We now record the number of connected components in the first few fixed-support strata.

\begin{center}
\small
\begin{longtable}{rcccccc}
\toprule
$n$ & comp$(G_n^{[1]})$ & comp$(G_n^{[2]})$ & comp$(G_n^{[3]})$ & comp$(G_n^{[4]})$ & comp$(G_n^{[5]})$ & comp$(G_n^{[6]})$\\
\midrule
\endfirsthead
\toprule
$n$ & comp$(G_n^{[1]})$ & comp$(G_n^{[2]})$ & comp$(G_n^{[3]})$ & comp$(G_n^{[4]})$ & comp$(G_n^{[5]})$ & comp$(G_n^{[6]})$\\
\midrule
\endhead
1 & 1 & 0 & 0 & 0 & 0 & 0 \\
2 & 1 & 0 & 0 & 0 & 0 & 0 \\
3 & 2 & 1 & 0 & 0 & 0 & 0 \\
4 & 3 & 1 & 0 & 0 & 0 & 0 \\
5 & 2 & 1 & 0 & 0 & 0 & 0 \\
6 & 4 & 1 & 1 & 0 & 0 & 0 \\
7 & 2 & 1 & 1 & 0 & 0 & 0 \\
8 & 4 & 2 & 1 & 0 & 0 & 0 \\
9 & 3 & 3 & 1 & 0 & 0 & 0 \\
10 & 4 & 4 & 1 & 1 & 0 & 0 \\
11 & 2 & 7 & 1 & 1 & 0 & 0 \\
12 & 6 & 10 & 1 & 1 & 0 & 0 \\
13 & 2 & 11 & 1 & 1 & 0 & 0 \\
14 & 4 & 16 & 1 & 1 & 0 & 0 \\
15 & 4 & 18 & 1 & 1 & 1 & 0 \\
16 & 5 & 24 & 1 & 1 & 1 & 0 \\
17 & 2 & 24 & 1 & 1 & 1 & 0 \\
18 & 6 & 33 & 2 & 1 & 1 & 0 \\
19 & 2 & 33 & 3 & 1 & 1 & 0 \\
20 & 6 & 42 & 3 & 1 & 1 & 0 \\
21 & 4 & 42 & 5 & 1 & 1 & 1 \\
22 & 4 & 53 & 9 & 1 & 1 & 1 \\
23 & 2 & 50 & 11 & 1 & 1 & 1 \\
24 & 8 & 69 & 16 & 1 & 1 & 1 \\
25 & 3 & 57 & 23 & 1 & 1 & 1 \\
\bottomrule
\end{longtable}
\end{center}

This table suggests a sharp qualitative split between very low support and moderate or high support.

\begin{itemize}
\item The support-$1$ stratum is exceptional and almost totally disconnected, in agreement with Proposition~6.3.
\item The support-$2$ stratum fragments heavily and remains highly disconnected throughout the computed range.
\item The support-$3$ stratum is connected up to $n=17$, but disconnects from $n=18$ onward in the data.
\item The support-$4$, $5$, and $6$ strata are connected throughout the computed range in which they are nontrivial.
\end{itemize}

At present, all of these are computational observations, not theorems.

\subsection{Representative support summaries at fixed $n$}

To see how several support statistics interact at once, it is helpful to compress the data by support level at a fixed value of $n$. For $n=20$, one obtains:

\[
\begin{array}{c|ccccc}
r & |V_{20,r}| & e_{20,r,r} & \mathrm{comp}(G_{20}^{[r]}) & \min\degG & \max\degG\\
\hline
1 & 6   & 0   & 6  & 1  & 2\\
2 & 81  & 41  & 42 & 3  & 6\\
3 & 277 & 606 & 3  & 7  & 11\\
4 & 227 & 942 & 1  & 13 & 17\\
5 & 36  & 120 & 1  & 21 & 23
\end{array}
\]

This one table already illustrates three important support phenomena.

\begin{enumerate}[label=\arabic*.]
\item \emph{Degree scale rises sharply with support.}
The observed minimum degrees
\[
1,3,7,13,21
\]
are compatible with the theorem-level floor
\[
\degG(\lambda)\ge r(r-1),
\]
and in this nonminimal range are in fact very close to $r(r-1)+1$.

\item \emph{Internal corridor density improves with support.}
The low-support strata are sparse and disconnected, while the higher-support strata are both denser and more coherent.

\item \emph{The main mass of the graph sits at intermediate support.}
For $n=20$, the vertex counts are concentrated at $r=3$ and $r=4$, exactly where the strongest inter-level edge traffic also occurs.
\end{enumerate}

\subsection{First occurrences}

For convenience, we collect several first-occurrence data.
\begin{enumerate}[label=\arabic*.]
\item \emph{Support size $r$} first appears at
\[
n=T_r,
\]
as proved in Section~2.

\item \emph{A nontrivial internal edge in $V_{n,r}$} first appears at
\[
n=2 \quad\text{for }r=1,
\]
and at
\[
n=T_r+1 \quad\text{for }r\ge 2,
\]
in agreement with Proposition~6.4.

\item \emph{Jump magnitude $2$} first appears at
\[
n=6.
\]

\item \emph{The first nontrivial disconnection of $G_n^{[3]}$} appears in the computed data at
\[
n=18.
\]

\item \emph{The first computed ranges in which $G_n^{[4]}$, $G_n^{[5]}$, and $G_n^{[6]}$ are nontrivial and connected} are
\[
n\ge 10,\quad n\ge 15,\quad n\ge 22,
\]
respectively, throughout the explored range.
\end{enumerate}

Again, only items (1) and (2), together with the trivial jump bound behind (3), belong to the theorem-level structure. The rest are empirical.

\subsection{Empirical remarks and cautious conjectural directions}

The atlas suggests several patterns that appear plausible but are not yet proved.

\begin{remark}[Empirical pattern A]
Very low support is structurally exceptional. In particular, support levels $1$ and $2$ seem to behave more like sparse arithmetic fringes than like the main corridor zone of the graph.
\end{remark}

\begin{remark}[Empirical pattern B]
Intermediate and upper-intermediate support levels carry most of the internal edge density and most of the large connected support-preserving regions.
\end{remark}

\begin{remark}[Empirical pattern C]
For fixed $n$, the coarse support traffic is concentrated near the diagonal and seems to peak in a middle band rather than at the extremes.
\end{remark}

These lead naturally to the following cautious problems.

\begin{problem}
Determine, for fixed $r$, the connectivity threshold of the induced graph $G_n^{[r]}$, if such a threshold exists.
\end{problem}

\begin{problem}
Describe the asymptotic location of the bulk of the support distribution
\[
a_{n,r}=|V_{n,r}|.
\]
\end{problem}

\begin{problem}
Study whether the main support corridor zone correlates with the zones of large degree, large local simplex dimension, or rear-central thickening.
\end{problem}

\begin{problem}
Develop refined support-jump invariants beyond the scalar quantity $\Delta_\sigma$, for example set-valued and profile-valued support jumps.
\end{problem}

\section{Conclusion and open problems}

The present paper turns support into a usable global language for the partition graph. The main theorem-level points are:

\begin{enumerate}[label=\arabic*.]
\item support size $r$ occurs exactly when $T_r\le n$;
\item support jumps along edges are exactly controlled by a birth--death rule and always lie in
\[
\{-2,-1,0,1,2\};
\]
\item support gives a universal quadratic floor for degree;
\item support size is preserved by conjugation;
\item the coarse support-level graph always contains the chain
\[
1-2-\cdots-\rho(n).
\]
\end{enumerate}

At the same time, the computational atlas shows that the fine geometry of the fixed-support strata is already rich and nontrivial. In particular, very low support appears structurally exceptional, while intermediate support strata seem to exhibit much more coherent corridor behaviour.

The next natural problems are:
\begin{enumerate}[label=\arabic*.]
\item determine the precise structure and connectivity of $G_n^{[r]}$;
\item study joint distributions of $(\sigma,\degG)$, $(\sigma,\dim_{\mathrm{loc}})$, and $(\sigma,d_{\mathrm{ax}})$;
\item locate support strata relative to framework, center, and rear-central thickening;
\item refine scalar support jumps to set-valued and profile-valued jump invariants;
\item develop a genuine support-gradient language on directed or height-filtered versions of $G_n$.
\end{enumerate}

These questions belong naturally to the next stage of the project, where jump-type and gradient-type invariants will be studied in their own right.

\section*{Acknowledgements}

The author acknowledges the use of ChatGPT (OpenAI) for discussion, structural planning, and editorial assistance during the preparation of this manuscript. All mathematical statements, proofs, computations, and final wording were checked and approved by the author, who takes full responsibility for the contents of the paper.

\end{document}